\documentclass[12pt,twoside,a4paper,reqno]{amsart}
\usepackage{bulletinUPB_3}
\usepackage{eucal}
\usepackage{amsmath, amsthm, amscd, amsfonts, amssymb, graphicx, color}
\usepackage[ansinew]{inputenc}
\usepackage[bookmarksnumbered, colorlinks, plainpages]{hyperref}
\usepackage{graphicx}
\usepackage{color}
\newtheorem{thm}{Theorem}[section]
\newtheorem{lem}[thm]{Lemma}
\newtheorem{prop}[thm]{Proposition}
\newtheorem{cor}[thm]{Corollary}
\theoremstyle{definition}
\newtheorem{defn}[thm]{Definition}

\theoremstyle{remark}

\newcommand{\h}{\mathcal{H}}

\newcommand{\X}{\mathrm{X}}
\newcommand{\s}{\quad }
\info{A}{--}{--}{201-}
\title{ Adjoint of Pair Frames}
\author{Abolhassan Fereydooni}
\address{$^1$Department of Mathematics, Vali-e-Asr University of Rafsanjan, Rafsanjan, Iran, e-mail: {\tt fereydooniman@yahoo.com }}
\author{Ahmad Safapour}
\address{$^2$Department of Mathematics, Vali-e-Asr University of Rafsanjan, Rafsanjan, Iran,   e-mail: {\tt safapour@vru.ac.ir }}
\author{Asghar Rahimi}
\address{$^3$Department of Mathematics, University of Maragheh, Maragheh, Iran,  e-mail: {\tt asgarrahimi@yahoo.com  }}

\begin{document}
\pagestyle{headings}
\maketitle

\begin{abstract}
\mm \mm
{\it \quad\qu  The concept of $(p,q)$-pair frames is generalized to $(\ell,\ell^*)$-pair frames. Adjoint (conjugate) of a pair frames for dual space of a Banach space is introduced and some conditions for the existence of adjoint (conjugate) of pair frames are presented.}
\end{abstract}
\begin{MSC2000}
42C\,15.
\end{MSC2000}

\section{ \textbf{Introduction} }

Frames were introduced by Duffin and Schaeffer \cite{1951 NonHor Duff}, in studying nonharmonic Fourier series. After some decades, Young  reintroduced frames in abstract Hilbert spaces  \cite{1980 NonHor Young}. Daubechies, Grossmann and Meyer studied frames deeply in 80's \cite{1986 Painless Non-Orthogonal Expansions- Daubechies A. Grossmann}. Feichtinger and Gröchenig \cite{1988 A unified approach to atomic decomp- Feichinger K. Grochenig, 1991 Atomic decompositions- Grochenig} extended the concept of frames from Hilbert spaces to Banach spaces and defined atomic decomposition and Banach frames.

The interested readers can refer to ~\cite{2012 Banach Pair Frames- A. Fereydooni A. Safapour} to study a memoir about \emph{frames from nonharmonic Fourier series to Banach pair frames}.

\emph{Pair frames} were introduced by the authors in Hilbert spaces \cite{2011 Pair Frames- A. Fereydooni A. Safapour}. They also considered pair frames in Banach spaces and defined Banach pair frames \cite{2012 Banach Pair Frames- A. Fereydooni A. Safapour}. It is shown that this notion generalizes some various types of frames. Some characterizations of Banach pair frames are presented in \cite{2012 Characterizations  for Banach Pair Frames- A. Fereydooni A. Safapour}.

The present paper is organized as follows. In section 2, some notations and required definitions are recalled. The concept of  frames and some types of frames in Banach spaces are considered in Section 3. It is proved that if we have two $\ell$-Bessel and $ \ell^*$-Bessel for a Banach space and its dual respectively, they are pairable.


If we have a (Banach) pair frame for a Banach space, a natural question can  arise:  Can one construct a (Banach) pair frame for the dual space using  this (Banach) pair frame?  Section 4 is devoted to address this question. Considering this subject, instead of the dual of pair frames, the concept of adjoint of pair frames arises.

Almost all propositions and theorems are stated in two cases, unconditional or nonunconditional cases. In each proposition and theorem the  unconditional case is  put in  "( )".
\section{ \textbf{Preliminaries} }

Through this paper, $\mathrm{X}$ ( $\mathcal{H}$ ) will denote a reflexive  Banach space (Hilbert space, rep.). $\langle.,.\rangle$ is used for the inner product of Hilbert spaces or the action  the functionals $\mathrm{X}^*$ on $\mathrm{X}$. "$\mathbb{I}$"  denotes  the index set of the natural numbers. The notation $\sigma$ is standed for permutations of $\mathbb{I}$.  All norms are denoted by $\|.\|;$ the reader can recognize conveniently that  to which concept each norm refers.

Let $L$ be a bounded linear operator, $\mathcal{D}(L)$ and  $ \mathcal{R}(L)$, denote the domain and the range of $L$, respectively. If $0<p<\infty$, and $\frac{1}{p}+\frac{1}{q}=1$, $q$ is called the exponential conjugate of $p$. \\
For $F=\{f_i\}\subset \mathrm{X}$ and $G=\{g_i\}\subset \mathrm{X^*}$, the operator
\[ U (U_{G}):\mathrm{X}\rightarrow {\mathbb{C}}^{\mathbb{I}}, \quad U(f):=\{\langle f , g_i \rangle\}, \]
is  called the \textbf{analysis operator}, and the operator
\[T (T_{F}):\mathcal{D}(T) ( \subset {\mathbb{C}}^{\mathbb{I}} ) \rightarrow \mathrm{X} , \quad  T(\{c_{i}\}):= \sum c_i f_i, \]
is said to be the  \textbf{synthesis operator}.


Let $\ell$ be a \textbf{Banach scalar sequence space}; a normed  vector space of scalar sequences which is a Banach space with respect to its norm.
If $\ell$ is a scalar sequence space, for every permutation $\sigma$ of $\mathbb{I}$, put
\[ \sigma F = \{f_{\sigma(i)} \}, \quad \sigma G = \{g_{\sigma(i)} \}, \quad \sigma \ell=\{ \{c_{\sigma(i)}\} \mid \{c_i\} \in \ell  \}.\]

  \begin{defn}\label{d:Unconditional scalar Sequence}
 We say that a Banach scalar sequence space  $\ell$ is an \textbf{unconditional Banach scalar sequence space} if for every permutation $\sigma$ of $\mathbb{I}$ and every  $\{c_i\} \in \ell$,
 \[\{c_{\sigma(i)}\} \in \ell, \s \|\{c_{\sigma(i)}\}\|=\| \{c_{i}\} \|.\]
 \end{defn}
A Banach scalar sequence space $\ell$ is called a \textbf{BK-space} if the coordinate functionals are continuous. Put ${\delta}_i=\{\delta_{ij}\}_{j} $ for $ i \in \mathbb{I}$, where $\delta_{ij} $ is the Kroneker delta for $ i,j \in \mathbb{I}$. $\{{\delta}_i\}$ is called the set of the \textbf{canonical vectors}. Additionally, when $\{{\delta}_i\}$ constitute a basis for $\ell$, $\ell$ is said to be a \textbf{Schauder sequence space }(CB-space or model space). Furthermore, when $\ell$ is reflexive, it is called an \textbf{RCB-space}.


For the proof of the next Lemma we refer to ~\cite[p.~201]{1964 Functional Analysis in Norm- Kantorovich}.


\begin{lem}\label{l: l Schauder sequence space then l  Schauder sequence sp}
Let $\ell$ be a Schauder sequence space. $\ell^*$, the dual of $\ell$, is isometrically isomorphic to a BK-space
\[ \ell^ \circledast =\{\{\phi ({\delta}_i )\} \mid \phi \in \ell^*\}.\]
 Also, for every linear functional $\phi \in \ell^*  $ there is a unique $\{d_i\}\in \ell^ \circledast$, so that $\phi$ has the form
\[ \phi (\{c_i\})=\sum d_i c_i, \quad \forall \{c_i\} \in \ell.\]
 The sequence $\{d_i\}$ is uniquely determined by $d_i=\phi({\delta}_i) $ for all $i \in \mathbb{I}$. Moreover if $\ell$ is reflexive, than $\ell^*$ is a Schauder sequence space.
\end{lem}
 $\ell^*$ and $\ell^ \circledast$ are identified in this paper.

\section{ \textbf{Frames in Hilbert and Banach Spaces} }

Here we restate some definitions and results from \cite{2012 Banach Pair Frames- A. Fereydooni A. Safapour}. The reader can refer there for considering the proofs.
 \begin{defn}\label{d:frame}
A family $F=\{f_i\}\subset \mathcal{H}$ is called a \textbf{frame} for $\mathcal{H}$ if there exist $A,B>0$ such that for every $f\in \mathcal{H}$,
\begin{equation}\label{e:frame}
A\| f \|^2 \leqslant \sum |  \langle f,f_i \rangle |^2 \leqslant B \| f \|^2 .
\end{equation}
If the right inequality is satisfied for some $B>0$,  $F=\{f_i\}$ is called a \textbf{Bessel sequence} for $\mathcal{H}$.
 \end{defn}
The following proposition is proved in \cite{2011 Pair Frames- A. Fereydooni A. Safapour}.
\begin{prop}\label{p: Bessels-pair Bessels}
A family $F=\{f_i\}\subset \mathcal{H}$ is a Bessel sequence for $ \mathcal{H}$ if and only if the operator
\begin{equation}\label{e:Bessel Operator}
S:\h \rightarrow \h, \s S(f )=\sum \langle f, f_i \rangle f_i,
\end{equation}
is a well defined operator. In this situation, $S$ is bounded.
 \end{prop}
\begin{thm}\label{t:frame}
(\cite{2012 Banach Pair Frames- A. Fereydooni A. Safapour})
A family $F=\{f_i\}\subset \mathcal{H}$ is a frame for $\mathcal{H}$ if and only if the operator $S$ defined in (\ref{e:Bessel Operator}) is  well defined and invertible.
\end{thm}
 As a standard reference about frame theory, \cite{2003 An introduction to frames- Christensen} can be suggested.

 \begin{defn}\label{d:pair frame}
Let $F=\{f_i\}\subset \mathrm{X}$ and $G=\{g_i\}\subset \mathrm{X^*}$. The pair  $(G,F)$  is said to be an (unconditional) \textbf{pair Bessel }if the operator
\[ S(S_{FG}): \X \rightarrow \X, \s S(f)=   \sum \langle f, g_i\rangle f_i ,\]
is well defined (unconditionally); i.e. the series converges (unconditionally) for every $f\in \mathrm{X}$.
The (unconditional) pair Bessel $(G,F)$ is called an (unconditional)\textbf{ pair frame} when $S$ is invertible.
 \end{defn}
Let $F=\{f_i\}\subset \mathrm{X}$ and $G=\{g_i\}\subset \mathrm{X^*}$. By the term  \emph{"$F=\{f_i\}$ and $G=\{g_i\}$ are \textbf{pairable} for $\mathrm{X}$"}, we mean that $(G,F)$ is a pair Bessel for $\mathrm{X}$.

Next proposition is proved in \cite{2012 Banach Pair Frames- A. Fereydooni A. Safapour}. But it  can also be concluded from Proposition \ref{p: Bessels-pair Bessels} and Theorem \ref{t:frame}. Proposition \ref{p:Hilbert frame = pair frame} shows that the pair frames (Bessels) are generalizations of frames (Bessel sequences) in Hilbert spaces.

\begin{prop}\label{p:Hilbert frame = pair frame}
$F=\{f_i\}\subset \mathcal{H}$ is a  frame (Bessel sequence) for $\mathcal{H}$ if and only if $ ( F,F )$ is a pair frame (Bessel) for $\mathcal{H}$. In this case $ (F,F )$ is an unconditional pair frame (Bessel).
 \end{prop}
 \begin{defn}\label{d:Schauder frame}
Let $F=\{f_i\}\subset \mathrm{X}$ and $G=\{g_i\}\subset \mathrm{X^*}$. The (unconditional) pair Bessel $(G,F)$  is said to be an (unconditional) \textbf{Schauder frame} for $\X$ if for every $f \in \mathrm{X}$,
\[f= \sum \langle f, g_i\rangle f_i,\]
(and the sum converges unconditionally).
 \end{defn}
 Every Schauder frame can be considered as  a pair frame. But the associated operator $S$, have to be the identity operator.
\begin{defn}\label{d:l-bessel}
Let $G=\{g_i\}\subset \mathrm{X^*}$ and $\ell$ be a BK-space. $G=\{g_i\}$ is called an (unconditional)\textbf{ $\ell$-Bessel} for $\mathrm{X}$ with bound $B>0$, if   for every $f\in \ X$,
\begin{enumerate}
\item $\{\langle f , g_i \rangle\} \in \ell   \quad  (\{\langle f , g_{\sigma(i)} \rangle\} \in \ell,   \forall \sigma )$,
\item $ \| \{\langle f , g_i \rangle\} \|   \leqslant  B \|f  \|  \quad (\| \{\langle f , g_{\sigma(i)} \rangle\} \|   \leqslant  B \|f  \|,  \forall \sigma).$
\end{enumerate}
Additionally, if  for every $f\in \ X$,
\[ A  \|f  \|   \leqslant  \| \{\langle f , g_i \rangle\} \|  \quad (A  \|f  \|   \leqslant \|  \{\langle f , g_{\sigma(i)} \rangle\} \| ,  \forall\sigma), \]
for some  $A>0$, $G=\{g_i\}$ is said to be an (unconditional)\textbf{ $\ell$-frame} for $\mathrm{X}$. $A$ and $B$ are called lower and upper $\ell$-frame bounds, respectively.
 \end{defn}
It is obvious that $ \|U_G\|, \|U_{\sigma G}\|   \leqslant  B$ for every permutation $\sigma$ of $\mathbb{I}$. If  $\ell$ is an unconditional Banach scalar sequence space and $G=\{g_i\}\subset \mathrm{X^*}$ is an $\ell$-Bessel for $\mathrm{X}$, then $G=\{g_i\}$ is an unconditional \textbf{ $\ell$-Bessel} for $\mathrm{X}$ and $ \|U_G\| = \|U_{\sigma G}\|   \leqslant  B$ for every permutation $\sigma$ of $\mathbb{I}$.

If conditions (1) and (2) in the  Definition \ref{d:l-bessel} are satisfied for some  $G=\{g_i\}\subset \mathrm{X^*}$ and $\ell=\ell^p$, $G=\{g_i\}$ is called a \textbf{$p$-Bessel} for $\mathrm{X}$; also if the lower inequality holds for some $A>0$, it is said to be a \textbf{$p$-frame} for $\mathrm{X}$.

After definition of $\ell$-Bessels with a BK-space $\ell$, the notion of  pair frames (Bessels) w.r.t.  $\ell$ can be defined.
 \begin{defn}\label{d:pair Bessel w.r.t. l-  l-pair Bessel}
Let $\ell$ be a BK-space and $G=\{g_i\}\subset \mathrm{X^*}$ be an (unconditional) $\ell$-Bessel for $\mathrm{X}$. If there exists   $F=\{f_i\}\subset \mathrm{X}$ such that $(G,F)$ is an (unconditional) pair Bessel for $\mathrm{X}$; i.e. the operator
  \[ S(S_{FG}):\mathrm{X}\longrightarrow  \mathrm{X}, \quad S(f):=\sum \langle f,g_i \rangle\ f_i.\]
 is well defined (unconditionally), then $(G,F)$ is called an (unconditional) \textbf{pair Bessel for $\mathrm{X}$ w.r.t. $\ell$} or an (unconditional) $\ell$-pair Bessel for  $\mathrm{X}$.

 Assume that  $( G,F)$  is an (unconditional) pair Bessel for $\mathrm{X}$ w.r.t. $\ell$. If the  operator $S$ is invertible, $( G,F)$ is called an (unconditional) \textbf{pair frame for $\mathrm{X}$ w.r.t. $\ell$} or an (unconditional) $\ell$-pair frame for  $\mathrm{X}$.

 Furthermore, if for an $\ell$-pair frame (Bessel) $(G,F)$, there is a BK-space $\ell'$ such that $F=\{f_i\}$ is an $\ell'$-Bessel, $(G,F)$  is said to be an  \textbf{$(\ell,\ell')$-pair frame (Bessel)} for $\mathrm{X}$ or a pair frame (Bessel) for $\mathrm{X}$ w.r.t. $(\ell,\ell')$.
\end{defn}

 \begin{defn}\label{d:atomic decomposition}
Let $\ell$ be a BK-space, $F=\{f_i\}\subset \mathrm{X}$ and $G=\{g_i\}\subset \mathrm{X^*}$. $(G,F)$  is called an (unconditional) \textbf{atomic decomposition for $\mathrm{X}$ w.r.t. $\ell$}, if there are $A,B>0$ such that  for every $f\in \ X$,
\begin{enumerate}
\item $\{\langle f , g_i \rangle\} \in \ell   \s (\{\langle f , g_{\sigma(i)} \rangle\} \in \ell,    \forall \sigma ),$
\item $ A  \|f  \|   \leqslant   \| \{\langle f , g_i \rangle\} \|   \leqslant  B \|f  \| \s ( A  \|f  \|   \leqslant  \|  \{\langle f , g_{\sigma(i)} \rangle\} \|   \leqslant  B \|f  \|,  \forall \sigma)$,
\item $ f=\sum \langle f , g_i \rangle\ f_i \s $  (series converges unconditionally).
\end{enumerate}
\end{defn}
Additionally, if $F=\{f_i\}$ is an $\ell'$-Bessel, which $\ell'$ is a BK-space, the above atomic decomposition is said to be an $(\ell,\ell')$-atomic decomposition.


 \begin{defn}\label{d:unconditionally operator}
Let $\ell$ be an unconditional Banach scalar sequence space. The bounded operator $T:\ell \rightarrow \mathrm{X}$ is said to be an \textbf{unconditional operator} from $\ell$  into $\mathrm{X}$ if for every permutation $\sigma$ of $\mathbb{I}$, there is a bounded operator $T_\sigma :\ell \rightarrow \mathrm{X}$  so that
\[T_{\sigma} (\{c_{\sigma(i)}\})=T(\{c_i\}).\]

 \end{defn}
The notion of pair frame is extended by generalizing   the synthesis  operator.
\begin{defn}\label{d:Banach pair frame}
Let  $G=\{g_i\}\subset \mathrm{X^*}$ and $T: \mathcal{R}(U_G) \rightarrow \mathrm{X}$ be an operator. $(G ,T)$ is called a \textbf{ Banach pair Bessel} for  $\mathrm{X}$ if the operator
\[S(S_{TG}):\mathrm{X}\longrightarrow  \mathrm{X}, \quad S(f):=T(\{\langle f ,g_i \rangle\}), \]
is bounded.
If for every permutation $\sigma$ of $\mathbb{I}$, there is an operator $T_{\sigma}: \mathcal{R}(U_{ \sigma G}) \rightarrow \mathrm{X}$  such that
\[  S_{\sigma}(f):=T_{\sigma} (\{\langle f ,g_{\sigma(i)} \rangle\})=S(f) ,\]
 $(G,T)$ is said to be an \textbf{unconditional Banach pair Bessel} for  $\mathrm{X}$.

Let  $( G,T)$ be an (unconditional) Banach pair Bessel and $S$ be the  associated  operator. If $S$ is invertible, $(G,T)$ is said to be an (unconditional)\textbf{ Banach pair frame} for  $\mathrm{X}$.
 \end{defn}

By implementing  a BK-space $\ell$, another version of the above definitions \emph{w.r.t.  $\ell$ }is defined.

\begin{defn}\label{d:Banach pair frame  w.r.t. l}
Let $\ell$ be a BK-space and $G=\{g_i\}\subset \mathrm{X^*}$ be an $\ell$-Bessel for $\mathrm{X}$. If $T:\ell \rightarrow \mathrm{X}$ is a bounded operator, $( G,T)$ is referred to as a \textbf{ Banach pair Bessel for  $\mathrm{X}$ w.r.t. $\ell$} or a Banach  $\ell$-pair Bessel for  $\mathrm{X}$.

Define
\[S(S_{GT}):\mathrm{X}\longrightarrow  \mathrm{X}, \quad S(f):=T(\{\langle f , g_i \rangle\}). \]
Additionally suppose that $G=\{g_i\}\subset \mathrm{X^*}$ is  an unconditional $\ell$-Bessel for $\mathrm{X}$. If for every permutation $\sigma$ of $\mathbb{I}$  there is a bounded operator $T_{\sigma}:\ell \rightarrow \mathrm{X}$  such that,
\[  S_{\sigma}(f):=T_{\sigma} (\{\langle f , g_{\sigma(i)} \rangle\})=S(f) ,\]
 $(G,T)$ is called an \textbf{unconditional Banach pair Bessel for  $\mathrm{X}$ w.r.t. $\ell$}.

Let  $(G,T)$ be an (unconditional) Banach pair Bessel and $S$ be its associated   operator. If $S$ is invertible, $(G,T)$ is said to be an (unconditional)\textbf{ Banach pair frame for  $\mathrm{X}$ w.r.t. $\ell$} or an (unconditional) Banach  $\ell$-pair frame for  $\mathrm{X}$.
 \end{defn}

 \begin{defn}\label{d:unconditional Banach Frames}
 Let $\ell$ be a  BK-space,  $G=\{g_i\}\subset \mathrm{X^*}$ and $T:\ell \rightarrow \mathrm{X}$ be a bounded operator. $(G,T)$ is called an \textbf{(unconditional) Banach frame for  $\mathrm{X}$ w.r.t. $\ell$}  if there are $A,B>0$ such that  for every $f\in \ X$,
\begin{enumerate}
\item $\{\langle f , g_i \rangle\} \in \ell \s (\{\langle f , g_{\sigma(i)} \rangle\} \in \ell,   \forall \sigma )$,
\item $  A \|f\|   \leqslant \| \{\langle f , g_i \rangle\} \|   \leqslant  B  \|f  \| \s ( A  \|f  \|   \leqslant  \|  \{\langle f , g_{\sigma(i)} \rangle\} \|   \leqslant  B \|f  \|, \forall \sigma)$,
\item $f=T (\{ \langle f,g_i \rangle \})  $

(for every permutation $\sigma$ of $\mathbb{I}$,  there is a bounded operator $T_{\sigma}:\ell \rightarrow \mathrm{X}$ such that
\[ f=T_{\sigma} (\{ \langle f,g_{\sigma(i)} \rangle \}).\]
\end{enumerate}
We refer to $T_{\sigma}$ as a permutation of $T$ for the permutation $\sigma$ of $\mathbb{I}$.
 \end{defn}
 The all operators $S$, in the above definition of pair frames (Bessels) is called \textbf{pair frame (Bessel) operator}.

  The authors have proved in \cite{2011 Pair Frames- A. Fereydooni A. Safapour} that for the conjugate exponentials $p$ and $q$, if $F=\{f_i\} \subset \mathrm{X}$ is $p$-Bessel and $G=\{g_i\} \subset \mathrm{X^*}$ is $q$-Bessel for $\mathrm{X^*}$ and $\mathrm{X}$, respectively, then $F=\{f_i\}$ and $G=\{g_i\}$ are pairable for $\mathrm{X^*}$ and $\mathrm{X}$. In this situation, $( F,G) $ is called a $(p,q)$-pair frame for $\mathrm{X}$. In the other words, by letting $\ell= \ell^p$ and $\ell^*= \ell^q$, if $F=\{f_i\} \subset \mathrm{X}$ is $\ell$-Bessel and $G=\{g_i\} \subset \mathrm{X^*}$ is $\ell^*$-Bessel for $\mathrm{X^*}$ and $\mathrm{X}$ respectively, then  $( F,G) $ and $( G,F)  $ are pair Bessels for $\mathrm{X^*}$ and $\mathrm{X}$, respectively.
A natural question which  arises  is that does analogous  results hold for general Banach scalar sequence spaces $\ell$ and $\ell^*$? The next theorem provides an affirmative answer to this question.
The claim is proved in \cite{2011 Pair Frames- A. Fereydooni A. Safapour} for $p$-Bessels and $q$-Bessels by using the Holder's inequality. In the proof of Theorem \ref{t: (l l star) pair Bessel } we don't use the Holder's inequality. Following the above nomination, we call such a pair frame (Bessel), \textbf{$(\ell  , \ell ^* )$-pair frame (Bessel)}.
\begin{thm}\label{t: (l l star) pair Bessel }
Let $\ell$ be a Schauder sequence space, $G=\{g_i\} \subset \mathrm{X^*}$ be an $\ell$-Bessel for $\mathrm{X}$  and $F=\{f_i\} \subset \mathrm{X}$ be an  $\ell^*$-Bessel for $\mathrm{X^*}$. Then $( G,F)  $ is a pair Bessel for $\mathrm{X}$ w.r.t. $\ell$.
\end{thm}

\begin{proof}
For $m,n \in  \mathbb{N}$ with  $n < m$ and for any $f \in \mathrm{X}$,
\[  \| \sum_{i=n}^{m} \langle f , g_i \rangle f_i \| = \sup_{\|g\|=1, g \in \mathrm{X^*}} | \sum_{i=n}^{m} \langle f , g_i \rangle  \langle f_i, g \rangle | .\]
Lemma \ref{l: l Schauder sequence space then l  Schauder sequence sp} yields that for every $g \in \mathrm{X^*}$ there is a $\phi_g \in \ell^* $ such that ${ \{   \langle f_i, g   \rangle  \}}_{i=1}^{\infty}$ as an element of $\ell^* $ can be rewritten in the form of $ \{\phi_g( \delta_i)\}_{i=1}^{\infty}$.
Hence
\begin{align*}
\| \sum_{i=n}^{m} \langle f , g_i \rangle f_i \|
& =  \sup_{\|g\|=1, g \in \mathrm{X^*}}  | \langle \sum_{i=n}^{m} \langle f, g_i\rangle f_i , g \rangle | \\
& =  \sup_{\|g\|=1, g \in \mathrm{X^*}}  |  \sum_{i=n}^{m} \langle f, g_i\rangle \langle f_i , g \rangle |   \\
& =  \sup_{\|g\|=1, g \in \mathrm{X^*}}  | \langle  { \{ \langle f , g_i \rangle \} }_{i=n}^{m},   { \{ \langle  f_i, g   \rangle \} }_{i=n}^{m}      \rangle |   \\
& =  \sup_{\|g\|=1, g \in \mathrm{X^*}}  | \langle  { \{ \langle f , g_i \rangle \} }_{i=n}^{m},   { \{ \langle f_i, g     \rangle \} }_{i=1}^{\infty}   \rangle |   \\
& =  \sup_{\|g\|=1, g \in \mathrm{X^*}}  | \phi_g ( { \{ \langle f , g_i \rangle \} }_{i=n}^{m} )|   \\
&   \leqslant   \sup_{\|g\|=1, g \in \mathrm{X^*}}  \| \phi_g \| \| { \{ \langle f , g_i \rangle \} }_{i=n}^{m}\|  \\
& =  \sup_{\|g\|=1, g \in \mathrm{X^*}}   { \| { \{ \langle f_i, g \rangle \}}_{i=1}^{\infty} \|}_{\ell^*}  \| { \{ \langle f , g_i \rangle \} }_{i=n}^{m} \| \\
&   \leqslant ( \sup_{\|g\|=1, g \in \mathrm{X^*}}  B \| g \|)  \| { \{ \langle f , g_i \rangle \} }_{i=n}^{m} \| \\
&  = B \| { \{ \langle f , g_i \rangle \} }_{i=n}^{m} \|,
\end{align*}
where $B$ denotes the upper Bessel bound of  $F=\{f_i\}$. The last value tends to zero when m and n tends to infinity. Therefore $ \sum_{}^{} \langle f , g_i \rangle f_i $ converges for every $ f \in \mathrm{X}$.
\end{proof}
 \begin{cor} \label{c:}
With the assumptions of the above theorem, if additionally $\ell^*$ is a Schauder sequence space, then $( F,G)  $ is a pair Bessel for $\mathrm{X^*}$.
\end{cor}

  \section{ \textbf{Frames for Dual Banach Spaces} }

It may seem that when $(G,F)$ is a pair Bessel (frame) for $ \mathrm{X}$, one can  conclude that $(F,G)$ is a pair Bessel (frame) for $ \mathrm{X^*}$. This is not true even for reflexive spaces  $ \mathrm{X}$. Furthermore there are examples for which $F=\{f_i\}$ and $G=\{g_i\} $ are not pairable for $\mathrm{X^*}$ even in Hilbert space setting; see Example 4.2 \cite{2008 The reconstruction property in Banach spaase Casazza and Ole Christensen}.

At the continue   we study some conditions under which a (Banach) pair Bessel (frame) for $ \mathrm{X}$   induces a (Banach) pair Bessel (frame) for $ \mathrm{X^*}$. But at first we state  some lemmas.
\begin{lem} \label{l:  U from X to l and l Bessels}
Let $\ell$ be an (unconditional) BK-space and $U:  \mathrm{X}\rightarrow \ell$ be a  bounded operator. Then there is an (unconditional) $\ell$-Bessel $H=\{h_i\} \subset \mathrm{X^*} $ such that $ U(f)= \{ \langle f , h_i \rangle\}$ for every $f \in  \mathrm{X}$. The Bessel bound of  $H=\{h_i\}$ is $\| U \|$.
\end{lem}
\begin{proof}
We prove the unconditional case; the proof  of general  case is in a similar way. Let $\sigma$ be a permutation of $\mathbb{I}$.
Since $\ell$ is a BK-space, the coordinate functionals $\{ \eta_i\} \subset \ell^*$ are continuous. Thus $h_{\sigma (i)}:= U^* \eta_{\sigma (i)} $'s are bounded functionals for all $i \in \mathbb{I}$. Hence  $ \{h_{\sigma (i)}\} \subset  \mathrm{X^*}$. For $f \in \mathrm{X}$,
\[ U_{}(f)= \{ \langle U_{}f , \eta_{i} \rangle \} = \{ \langle f , U^* \eta_{i} \rangle \} = \{ \langle f , h_{i} \rangle \} .\]
and
$$ \{  \langle f , h_{\sigma(i)} \rangle \}= \{ \langle f, U^* \eta_{\sigma(i)} \rangle \}= \{ \langle  Uf, \eta_{\sigma(i)}  \rangle \} .$$
On the other hand since $\ell$ is an unconditional BK-space
\[ \begin{split}
 \| \{  \langle f , h_{\sigma(i)} \rangle \} \|= \|\{ \langle  Uf, \eta_{\sigma(i)}  \rangle \}  \|=\| \{ \langle  Uf, \eta_{i}  \rangle \} \|=\|  Uf \|.
\end{split} \]
Then for all $f \in \X$ and permutation $\sigma$ of $\mathbb{I}$,
 \[\| \{ \langle f , h_{\sigma (i)} \rangle \} \|= \| U_{ }(f) \|   \leqslant   \| U\| \| f \|  .\]
Hence $ H=\{h_i\} \subset$ is an unconditional $\ell$-Bessel for $\mathrm{X}$ with bound $\| U\|$.

\end{proof}
\begin{lem} \label{l: weak convergence unconditional convergence}
\cite{1998 Basis Theory Primer Heil,1993 Uber unbedingte Orlicz,1938 On Integration Pettis}
For a sequence $H=\{h_i\}\subset\mathrm{X}$, the followings are equivalent.
\begin{enumerate}
\item $\sum h_i $ converges unconditionally.
\item $\sum h_{i_{k}} $ converges  for every $ \{h_{i_{k}}\} \subset \{h_i\}. $
\item $\sum h_{i_{k}} $ converges weakly for every $ \{h_{i_{k}}\} \subset \{h_i\}. $
\end{enumerate}
\end{lem}
\begin{thm}\label{t:adjoint of pair frames(Bessel)}
Let $\ell$ be a BK-space, $F=\{f_i\} \subset \mathrm{X}$ and  $G=\{g_i\} \subset \mathrm{X^*}$. Also assume that $T: \ell \rightarrow \mathrm{X}$ is a bounded operator.
 \begin{enumerate}
 \item  If  $(G,F)$ and $(F,G)$ are pair Bessels for $ \mathrm{X}$ and $ \mathrm{X^*}$ respectively, then $S_{FG}^{*}= S_{GF}$. In this situation, $(G,F)$ is a pair frame for $ \mathrm{X}$ if and only if $(F,G)$ is a pair frame for $ \mathrm{X^*}$.
 \item  $( G,F)$ is an unconditional pair frame (Bessel) for $ \mathrm{X}$ if and only if $(F,G)$ is an unconditional pair frame (Bessel) for $ \mathrm{X^*}$. Then $S_{FG}^{*}= S_{GF}$.
\item  Suppose that $(G,T)$ is a  Banach pair frame (Bessel) for $ \mathrm{X}$ w.r.t. $\ell$. Then there exists a family $H=\{h_i\}\subset\mathrm{X^*}$ such that $(H ,U_G^*)$ is a  Banach pair frame (Bessel) for $ \mathrm{X^*}$ w.r.t. $\ell^*$ with pair frame (Bessel) operator  $U_G^* U_H=S_{TG}^*$.
\item Let $\ell$ and $\ell^*$ be Schauder sequence spaces. $(G.F)$ is an $(\ell,\ell^*)$-pair frame (Bessel) for $ \mathrm{X}$ if and only if $(F,G)$ is an $(\ell^*,\ell)$-pair frame (Bessel) for $ \mathrm{X^*}$. Then $S_{FG}^{*}= S_{GF}$.
 \end{enumerate}
\end{thm}
\begin{proof}
The pair Bessel case of the claims are included  here. Since the invertibility of adjoin of an operator is equivalent to the invertibility of the  operator, itself, the pair frame case of the claims can be concluded  conveniently.\\
(1). Since $(G,F)$ and $(F,G)$ are pair Bessels for $ \mathrm{X}$ and $ \mathrm{X}^*$ respectively, then $S_{FG}$ and $S_{GF}$ are well defined and for $f \in \mathrm{X}$ and $g \in \mathrm{X^*}$,
\[ \begin{split}
 \langle S_{FG} (f)  , g \rangle = \langle  \sum \langle  f , g_i \rangle f_i , g \rangle = \sum \langle  f , g_i \rangle \langle f_i  , g \rangle = \sum \langle  f , g_i \rangle  \langle f_i  , g \rangle \\ = \langle f  , S_{GF}(g) \rangle.
\end{split}
 \]
 Thus $ S_{FG}^*=  S_{GF} $.\\
(2). We restate  the proof from  \cite{2011 Pair Frames- A. Fereydooni A. Safapour}. The Lemma \ref{l: weak convergence unconditional convergence} is used frequently. $(G,F)$ is an unconditional pair Bessel for $ \mathrm{X}$ if and only if $ \sum \langle  f , g_i \rangle f_i$ converges unconditionally for all $ f \in \mathrm{X}$. This is equivalent to the weak convergence of each  of  its subseries. Equivalently, for  every  $ \{f_{i_k} \} \subset \{f_i\} $  and $ \{g_{i_k} \} \subset \{g_i\} $,
\[  \langle \sum \langle  f , g_{i_k} \rangle f_{i_k} , g \rangle =   \langle  f , \sum  \langle f_{i_k}  , g \rangle g_{i_k} \rangle , \]
for $f \in \mathrm{X}$ and $g \in \mathrm{X^*}$. This means that $ \sum  \langle f_{i_k}  , g \rangle g_{i_k}$ converges weakly for every $g \in \mathrm{X^*}$. Namely $ \sum  \langle f_{i}  , g \rangle g_{i}$ converges unconditionally for all $g \in \mathrm{X^*}$. This leads to the fact that $(F,G)$ being  an unconditional pair Bessel for $ \mathrm{X^*}$.\\

(3). Suppose that  $(G,T)$ is a Banach pair Bessel for $ \mathrm{X}$ w.r.t. $\ell$. So we get bounded operators $T^*: \mathrm{X^*}  \rightarrow \ell^*$ and $U_G^*: \ell^* \rightarrow \mathrm{X^*}$.  $\ell^*$ is a BK-space, by  Lemma \ref{l: l Schauder sequence space then l  Schauder sequence sp}. Using  Lemma \ref{l:  U from X to l and l Bessels} for $\ell^*$ and the bounded operator $T^*: \X^* \rightarrow \ell^*$, the $\ell^*$-Bessel  $ H=\{h_i\} \subset  \mathrm{X^{**}}=\X$ can be obtained such that for every $g \in \mathrm{X^*}$,
\[ T^*(g)= \{ \langle g , h_i \rangle \}=U_H(g).\]
Therefore
$$S_{TG}^*={(TU_G)}^*= U_G^* T^*=U_G^* U_H$$
 and $(H ,U_G^*)$ is a Banach pair Bessel for $ \mathrm{X^*}$ w.r.t. $\ell^*$ with the pair Bessel operator $U_G^* U_H=S_{TG}^*$.\\ 
(4). Assume  that $(G,F)$ is an $(\ell,\ell^*)$-pair Bessel for $ \mathrm{X}$. Then $G=\{g_i\}$ and $F=\{f_i\}$ are Bessels w.r.t. $\ell$ and $\ell^*$,  respectively. Consequently by Theorem \ref{t: (l l star) pair Bessel },  $( F,G)$ is an $(\ell^*,\ell)$-pair Bessel for $ \mathrm{X^*}$. The proof of the converse is the same.
\end{proof}
\begin{cor} \label{c:}
Let $\ell$ be a BK-space, $F=\{f_i\} \subset \mathrm{X}$ and  $G=\{g_i\} \subset \mathrm{X^*}$. Assume that  $T: \ell \rightarrow \mathrm{X}$ is a bounded operator.
 \begin{enumerate}
 \item  Suppose that $(G,F)$ and $(F,G)$ are pair Bessels for $ \mathrm{X}$ and $ \mathrm{X^*}$ respectively. Then $(G,F)$ is a Schauder frame for $ \mathrm{X}$ if and only if $(F,G)$ is a Schauder frame for $ \mathrm{X^*}$.

\item  $(G,F)$ is an unconditional Schauder frame  for $ \mathrm{X}$ if and only if $(F,G)$ is an unconditional Schauder frame  for $ \mathrm{X^*}$
\item  Suppose that $(G,T)$ is a  Banach frame for $ \mathrm{X}$ w.r.t. $\ell$. Then there exists a family $H=\{h_i\}\subset\mathrm{X^*}$ such that $(H ,T_G)$ is a  Banach frame for $ \mathrm{X^*}$ w.r.t. $\ell^*$.

\item Let $\ell$ and $\ell^*$ be  Schauder sequence spaces. $(G,F)$ is an $(\ell,\ell^*)$-atomic decomposition for $ \mathrm{X}$ if and only if $(F,G)$ is an $(\ell^*,\ell)$-atomic decomposition for $ \mathrm{X^*}$.
\end{enumerate}
\end{cor}
\begin{proof}
   Only put $S=I$ in the Theorem \ref{t:adjoint of pair frames(Bessel)}.
\end{proof}
Considering the above arguments, we can speak about \textbf{adjoint} or \textbf{conjugate of pair frames} for the dual (conjugate) of a Banach spaces.

\begin{prop}\label{p:operatore and pair frames}
Let $\ell$ be a BK-space, $F=\{f_i\}\subset \mathrm{X}$ and $G=\{g_i\}\subset \mathrm{X^*}$. Suppose that $T:\ell \rightarrow \mathrm{X}$ is a bounded  operator and $V,W$ are bounded operators on  $\X$.
\begin{enumerate}
\item If $(G,F)$ ( $( G,  T)$ ) is an (unconditional) pair Bessel (Banach pair Bessel w.r.t. $\ell$), then $(\{W^* g_i\},\{V f_i\})$ ( $( \{ W^*g_i\} , V T)$ ) is an (unconditional) pair Bessel (Banach pair Bessel w.r.t. $\ell$).
\item If $V,W$ are invertible and $(G,F)$ ( $( G ,  T)$ ) is an (unconditional) pair frame (Banach pair frame w.r.t. $\ell$), then $(\{W^* g_i\},\{V f_i\})$ ( $( \{ W^*g_i\} , V T)$ ) is an (unconditional) pair frame (Banach pair frame w.r.t. $\ell$).
\end{enumerate}

 \end{prop}
\begin{proof}
We prove the assertion in the unconditional setting. Let $f \in \mathrm{X}$ and $\sigma$ be a permutation of $\mathbb{I}$. In the pair Bessel case, we have
\[ VS W (f)= VS_{\sigma }W (f)=\sum \langle W f , g_{\sigma  (i)} \rangle V f_{\sigma  (i)}= \sum \langle  f , W^* g_{\sigma  (i)} \rangle V f_{\sigma  (i)}  .\]
In the Banach pair Bessel case, let  $T_{\sigma } $ be a permutation of $T$. Then
\[ VSW (f)=  VS_{\sigma }W (f)=V T_{\sigma } (\{ \langle W f , g_{\sigma  (i)} \rangle \})= V T_{\sigma } (\{ \langle f ,  W^* g_{\sigma  (i)} \rangle \}). \]
  The above relations prove assertion (1). The  assertion (2) is a result of invertibility of $VSW $, when $V$ and $W$ are invertible.
\end{proof}

\bigskip

\end{document}